\documentclass[12pt,a4paper,reqno]{amsart}
\usepackage{amssymb,amsfonts,amsmath}
\usepackage[english,russian]{babel}
\addto\captionsrussian{}
\usepackage{color}
\usepackage[dvips]{graphicx}
\theoremstyle{plain}
\newtheorem{theorem}{Theorem}

\newtheorem{lemma}{Lemma}

\theoremstyle{definition}
\newtheorem{definition}{Definition}

\begin{document}
\title[On  local  linear convexity generalized  to commutative algebras]{On  local  linear convexity generalized  to commutative algebras}
\author{Tetiana M. Osipchuk}
\address{Institute of Mathematics of the National Academy of Sciences of Ukraine, Tereshchenkivska str. 3, UA-01004, Kyiv, Ukraine}
\email{osipchuk@imath.kiev.ua}
\begin{abstract}
В работе рассматривается коммутативная ассоциативная алгебра $\mathcal{A}$ над полем действительных чисел с единицей, которая обладает базисом $\{\boldsymbol{e}_k\}_{k=1}^{m}$, таким, что все его элементы $\boldsymbol{e}_k$ являются оборотными и среди матриц $\Gamma^p=(\gamma_{lk}^p)$, $p=\overline{1,m}$, где $\gamma_{lk}^p$ --- структурные константы $\mathcal{A}$ (то есть, $\boldsymbol{e}_l\boldsymbol{e}_k=\sum_{p=1}^{m}\gamma_{lk}^p\boldsymbol{e}_p$, $l,k=\overline{1,m}$),  существует по крайней мере одна невырожденная. Понятие линейно выпуклых областей в многомерном комплексном пространстве и некоторые их свойства обобщены на пространство $\mathcal{A}^n$, которое является  декартовым произведением $n$ алгебр $\mathcal{A}$.  А именно, получены необходимое и отдельно достаточное условия  $\mathcal{A}$-линейной выпуклости областей с гладкой границей в пространстве  $\mathcal{A}^n$ в терминах неотрицательности и соответственно положительности формальной квадратичной дифференциальной формы в  $\mathcal{A}$.
\vskip 2mm
\noindent 
 A commutative associative algebra $\mathcal{A}$ with an identity  over the field of real numbers which  has a basis $\{\boldsymbol{e}_k\}_{k=1}^{m}$, where all elements $\boldsymbol{e}_k$ are  invertible, is considered in the work. Moreover, among matrixes $\Gamma^p=(\gamma_{lk}^p)$, $p=\overline{1,m}$, consisting of the structure constants $\gamma_{lk}^p$ of $\mathcal{A}$, defined as $\boldsymbol{e}_l\boldsymbol{e}_k=\sum_{p=1}^{m}\gamma_{lk}^p\boldsymbol{e}_p$, $l,k=\overline{1,m}$,  there is at least one that is non-degenerate. The notion of linearly convex domains in the multi-dimensional complex space and some of their properties are generalized to the space $\mathcal{A}^n$ that is the Cartesian product of $n$ algebras $\mathcal{A}$.  Namely, the separate necessary and sufficient conditions of the local $\mathcal{A}$-linear convexity of domains with smooth boundary  in  $\mathcal{A}^n$ are obtained in terms of nonnegativity and positivity of  formal quadratic differential form in $\mathcal{A}$, respectively.
\end{abstract}


\maketitle
\section{Introduction}
The notion of linear convexity that  is studied in the theory of functions of many complex variables was coined in 1935 by H. Behnke and E. Peschl \cite{Ben1935}, but it has been actively used only since the 60s due to the works of A. Martineau \cite{Mar1966} and L. Aizenberg \cite{Aiz1967}, \cite{Aiz1967_2} who defined a linearly convex set in $n$-dimensional complex space $\mathbb{C}^n$ independently in slightly different ways. Here we give the Aizenberg's definition, since we take it as the basis. Hereinafter, a neighborhood of a point in a vector space is an open ball centered at this point.

\begin{definition}\label{def1}  (L. Aizenberg \cite{Aiz1967}) A domain
$D\subset\mathbb{C}^n$ is said to be \emph{\textbf{locally linearly convex}} if for every  boundary point
$w=(w_1,w_2\ldots ,w_n)\in\partial D$ there is a complex
hyperplane
\begin{multline*}
\Pi_{\mathbb{C}}(w):=\left\{z=(z_1,z_2,\ldots ,z_n)\in\mathbb{C}^n: \sum\limits_{j=1}^nc_j\left(z_j-w_j\right)=0,\right.\\
\Bigg.(c_1,c_2,\ldots ,c_n)\in\mathbb{C}^n\setminus\{0\}, \Bigg\}
\end{multline*}
 passing through $w$
but not intersecting $D$ in some neighborhood of the point $w$. If $\Pi_{\mathbb{C}}(w)\cap D=\varnothing$, then $D$
is said to be \emph{\textbf{(globally) linearly convex}}.
\end{definition}

However, H. Behnke and E. Peschl in \cite{Ben1935} considered linearly convex sets only in the complex plane   $\mathbb{C}^2$. They proved that global linear convexity follows the local one for bounded domains with a smooth boundary in $\mathbb{C}^2$.  For the case of $\mathbb{C}^n$ this result was obtained in 1971 by A. Yuzhakov and V. Kryvokolesko \cite{YuK1971}. Besides, in the work \cite{Ben1935}, the  analytical conditions of local linear convexity of domains with a smooth boundary in $\mathbb{C}^2$ (Behnke-Peschl conditions) were obtained.

In 1971 B. Zinoviev got the following generalization of Behnke-Peschl conditions for the case
$\mathbb{C}^n$, $n\ge 2$ \cite{Zin1971}.
Let domain
\begin{equation} \label{d}D=\{z=(z_1,\ldots, z_n)\in\mathbb{C}^n: \varphi(z)=\varphi(z,\bar{z})<0\}\end{equation}
be defined by the function
$\varphi(z):\mathbb{C}^{n}\to\mathbb{R}$, where $\varphi\in C^2$ in a neighborhood of the boundary $\partial D=\{z\in\mathbb{C}^n: \varphi(z)=0\}$  of $D$ and $\mathrm{grad}\varphi\ne 0$ everywhere on $\partial D$.
Then the following theorem is true.

\begin{theorem}\label{Zin} \textup{(\cite{Zin1971})} If a domain
$D$ is
 locally linearly convex, then for every boundary point $w\in\partial
 D$ and for all vectors $s=(s_1,s_2,\ldots ,s_n)\in\mathbb{C}^n$, $\|s\|=1$,  such that
$$
\sum\limits_{j=1}^{n}\frac{\partial\varphi (w)}{\partial
z_j}s_j=0
$$
the following inequality is true
\begin{equation} \label{1}
\sum\limits_{j,k=1}^{2n}\frac{\partial^2\varphi (w)}{\partial
z_j\partial z_k}s_js_k\ge 0,\quad\mbox{where}\quad
z_{n+j}=\bar{z}_j,\,\,s_{n+j}=\bar{s}_j, \,\, j=\overline{1,n}.
\end{equation}
If for every boundary point $w\in\partial D$ and for the same
vectors $s$
\begin{equation} \label{2}
\sum\limits_{j,k=1}^{2n}\frac{\partial^2\varphi (w)}{\partial
z_j\partial z_k}s_js_k> 0,\quad\mbox{where}\quad
z_{n+j}=\bar{z}_j,\,\,s_{n+j}=\bar{s}_j, \,\, j=\overline{1,n},
\end{equation}
then domain $D$ is locally linearly convex.
\end{theorem}

Similar conditions were obtained for the algebra of real quaternions \cite{Osi2006}, the algebra of real generalized quaternions \cite{Osi2013}, and Clifford algebras \cite{Osi2010}.  Let us notice here that the listed algebras are noncommutative. Professor A. Pogoruy reviewing these results proposed to obtain similar conditions for arbitrary commutative algebra. The purpose of the present work is to obtain necessary and sufficient conditions of generalized local linear convexity for a commutative associative algebra $\mathcal{A}$  over the field of real numbers with an identity and with some conditions imposed on its basis which are described in chapter 2. In that chapter real linear and quadratic forms are  presented in terms of  algebra $\mathcal{A}$ numbers and the generalization of the complex formal partial derivatives to algebra $\mathcal{A}$ is obtained. In chapter 3 the notion of linear convexity  and the conditions of local linear convexity (\ref{1})-(\ref{2}) are generalized to the space $\mathcal{A}^n$ that is the Cartesian product of $n$ algebras $\mathcal{A}$.

\section{Real linear and quadratic forms in commutative algebras}

In what follows, unless otherwise is specified, an $m\times n$ matrix of elements $a_{ij}\in\mathbb{R}$, $i=\overline{1,n}$, $j=\overline{1,n}$, will be denoted as $(a_{ij})$ and its determinant as $\mathrm{det}\,(a_{ij})$.
Let $\mathcal{A}$ be a commutative and associative algebra  over the field of real numbers $\mathbb{R}$ with identity $\boldsymbol{e}$. We identify $\boldsymbol{e}$ with $1$. Let $\mathrm{dim} \mathcal{A}=m$, elements $\{\boldsymbol{e}_k\}_{k=0}^{m-1}$ be a basis of $\mathcal{A}$, and $\gamma_{lk}^p\in\mathbb{R}$ be structure constants $\gamma_{lk}^p\in\mathbb{R}$ of $\mathcal{A}$  defined as follows:
\begin{equation}\label{struct}
\boldsymbol{e}_l\boldsymbol{e}_k=\sum\limits_{p=0}^{m-1}\gamma_{lk}^p\boldsymbol{e}_p\quad l,k=\overline{0,m-1}.
\end{equation}
Then each element $\boldsymbol{x}\in \mathcal{A}$ can be presented as 
\begin{equation}\label{struct2}
\boldsymbol{x}=\sum\limits_{q=0}^{m-1}x_q\boldsymbol{e}_q, \quad x_q\in\mathbb{R}.
\end{equation}
Numbers of algebra $\mathcal{A}$ will be denoted by small Latin letters in bold and the real numbers will be denoted by small Latin or Greek letters in normal font. Since a basis of some algebras includes identity and the other elements of the basis are denoted as $\boldsymbol{e}_1$, $\boldsymbol{e}_2$, etc., it is convenient to denote the identity as $\boldsymbol{e}_0$. So hereinafter, we start the numeration of the basis decomposition of $\mathcal{A}$  numbers  from zero. Such a numeration of the basis decomposition requires starting the numeration of the elements of matrixes and other objects within this paper from zero too.

Let the basis satisfy the following conditions:
\begin{enumerate}
\item[1)] there exist the inverse elements $\boldsymbol{e}_k^{-1}=\dfrac{1}{\boldsymbol{e}_k}$, $k=\overline{0,m-1}$.
\item [2)] among the matrixes $\Gamma^p=(\gamma_{lk}^p)$, $p=\overline{0,m-1}$ there is at least one that is non-degenerate.
\end{enumerate}
The author does not know whether  it is possible for any commutative and associative algebra  over $\mathbb{R}$ with an identity to choose a basis that simultaneously satisfies conditions 1) and 2).

Let us consider $n$-dimensional vector space
$$
\mathcal{A}^n:=\underbrace{\mathcal{A}\times
\mathcal{A}\times\ldots\times
\mathcal{A}}\limits_n
$$
with elements $\boldsymbol{z}=(\boldsymbol{z}_1,\boldsymbol{z}_2,\ldots,\boldsymbol{z}_n)\in \mathcal{A}^n$, where
\begin{equation}\label{eque5}
\boldsymbol{z}_j:=\sum\limits_{q=0}^{m-1}x^j_q\boldsymbol{e}_q\in
\mathcal{A},\quad x^j_q\in \mathbb{R},\,\,j=\overline{1,n}.
\end{equation}
We identify points (vectors) $\boldsymbol{z}\in \mathcal{A}^n$ with points (vectors) $z=(x^1_0,x^1_1,\ldots, x^n_{m-1})\in\mathbb{R}^{mn}$. Herewith, the elements of the space $\mathcal{A}^n$ are in a bold font and the elements of the space $\mathbb{R}^{mn}$ are in normal font.
Let
$$
\|\boldsymbol{z}\|=\sqrt{\sum\limits_{j=1}^n\sum\limits_{q=0}^{m-1}\left|x_q^j\right|^2}.
$$
Consider the following matrixes
$$
\boldsymbol{E}=\begin{pmatrix}
\boldsymbol{e}_0&0&\hdotsfor{1}&0 \\
0&\boldsymbol{e}_1&\hdotsfor{1}&0 \\
\vdots& \vdots &\ddots & \vdots  \\
0&0&\hdotsfor{1}&\boldsymbol{e}_{m-1} \\
\end{pmatrix},\quad
X^j=\left(\begin{array}c x^j_0\\ x^j_1 \\ \ldots\\ x^j_{m-1} \end{array}\right),\,\,j=\overline{1,n},
$$
and a non degenerate $m\times m$ matrix
$$
\Gamma=\begin{pmatrix}
1&1&\hdotsfor{1}&1 \\
\gamma_{10}&\gamma_{11}&\hdotsfor{1}&\gamma_{1(m-1)}\\
\vdots& \vdots &\ddots & \vdots  \\
\gamma_{(m-1)0}&\gamma_{(m-1)1}&\hdotsfor{1}&\gamma_{(m-1)(m-1)}
\end{pmatrix},\,\,\mbox{where}\,\,\gamma_{lq}\in\mathbb{R}.
$$
And let
\begin{equation}\label{eque6}
\boldsymbol{Z}_j=\Gamma \boldsymbol{E}X^j,
\end{equation}
where
$$
\boldsymbol{Z}_j=\left(\begin{array}c \boldsymbol{z}^{\boldsymbol{0}}_j\\ \boldsymbol{z}^{\boldsymbol{1}}_j \\ \ldots\\ \boldsymbol{z}^{\boldsymbol{m-1}}_j \end{array}\right),\,\,j=\overline{1,n}.
$$
Thus,
$$
\boldsymbol{z}_j^{\boldsymbol{l}}=\sum\limits_{q=0}^{m-1}\gamma_{lq}x^j_q\boldsymbol{e}_q,\,\,j=\overline{1,n},\quad \mbox{where} \,\, \gamma_{lq}=1 \,\, \mbox{as}\,\,l=0.
$$
From now on, for any number $\boldsymbol{x}\in\mathcal{A}$ the numbers $\boldsymbol{x}^{\boldsymbol{l}}\in\mathcal{A}$, $l=\overline{0,m-1}$, with upper index $l$ in bold are obtained from  $\boldsymbol{x}$ by multiplying the elements of  $l$th row of matrix $\Gamma$ by the respective  summands  $x_q\boldsymbol{e}_q$ in the basis decomposition (\ref{struct2}) of  $\boldsymbol{x}$. As we can see, $\boldsymbol{x}^{\boldsymbol{0}}=\boldsymbol{x}$.

We obtain from (\ref{eque6}):
$$
X^j=\boldsymbol{E}^{-1}\Gamma^{-1}\boldsymbol{Z}_j,
$$
where $\Gamma^{-1}=(\eta_{lp})$, $\eta_{lp}\in\mathbb{R}$, $l,p=\overline{0,m-1}$. That is to say,
$$
\left(\begin{array}c x^j_0\\ x^j_1 \\ \ldots\\ x^j_{m-1} \end{array}\right)=
\begin{pmatrix}
\boldsymbol{e}^{-1}_0&0&\hdotsfor{1}&0 \\
0&\boldsymbol{e}_1^{-1}&\hdotsfor{1}&0 \\
\vdots& \vdots &\ddots & \vdots  \\
0&0&\hdotsfor{1}&\boldsymbol{e}_{m-1}^{-1}
\end{pmatrix}
\begin{pmatrix}
\eta_{00}&\eta_{01}&\hdotsfor{1}&\eta_{0(m-1)} \\
\eta_{10}&\eta_{11}&\hdotsfor{1}&\eta_{1(m-1)}\\
\vdots& \vdots &\ddots & \vdots  \\
\eta_{(m-1)0}&\eta_{(m-1)1}&\hdotsfor{1}&\eta_{(m-1)(m-1)}
\end{pmatrix}
\left(\begin{array}c \boldsymbol{z}^{\boldsymbol{0}}_j\\ \boldsymbol{z}^{\boldsymbol{1}}_j \\ \ldots\\ \boldsymbol{z}^{\boldsymbol{m-1}}_j \end{array}\right).
$$
Hence
\begin{equation}\label{x}
x^j_l=\boldsymbol{e}^{-1}_l\sum\limits_{p=0}^{m-1}\eta_{lp}\boldsymbol{z}_j^{\boldsymbol{p}},\,\, j=\overline{1,n},\,\, l=\overline{0,m-1}.
\end{equation}

Consider a real linear form
$$
\sum\limits_{j=1}^n\sum\limits_{l=0}^{m-1}a_l^jx_l^j,
$$
where $a_l^j\in\mathbb{R}$, $a_l^j=\mathrm{const}$,
$j=\overline{1,n}$, $l=\overline{0,m-1}$.
Substitute $x_l^j$ for their expressions from (\ref{x}) and group together the respective components with  $\boldsymbol{z}_j^{\boldsymbol{p}}$, $j=\overline{1,n}$\,,
$l=\overline{0,m-1}$ fixing $j$ and $p$. Then we obtain
$$
\sum\limits_{j=1}^n\sum\limits_{l=0}^{m-1}a_l^jx_l^j=
\sum\limits_{j=1}^n\sum\limits_{l=0}^{m-1}a_l^j\boldsymbol{e}^{-1}_l\sum\limits_{p=0}^{m-1}\eta_{lp}\boldsymbol{z}_j^{\boldsymbol{p}}=
\sum\limits_{j=1}^n\sum\limits_{p=0}^{m-1}\boldsymbol{z}_j^{\boldsymbol{p}}\sum\limits_{l=0}^{m-1}\eta_{lp}a_l^j\boldsymbol{e}^{-1}_l=
\sum\limits_{j=1}^n\sum\limits_{p=0}^{m-1}\boldsymbol{z}_j^{\boldsymbol{p}}\boldsymbol{a}_j^p,
$$
where
\begin{equation}\label{Ajp}
\boldsymbol{a}_j^p=\sum\limits_{l=0}^{m-1}\eta_{lp}a_l^j\boldsymbol{e}^{-1}_l, \,\, j=\overline{1,n},\,\,   p=\overline{0,m-1}.
\end{equation}
Let us rewrite the expression of $\boldsymbol{a}_j^p$ in terms of indexes $i$, $q$.
$$
\boldsymbol{a}_i^q=\sum\limits_{k=0}^{m-1}\eta_{kq}a_k^i\boldsymbol{e}^{-1}_k, \,\, i=\overline{1,n},\,\,   q=\overline{0,m-1}.
$$

Now we consider a real quadratic form
$$
\sum\limits_{j,i=1}^n\sum\limits_{l,k=0}^{m-1}
a_{lk}^{ji}x_l^jx_k^i,
$$
where $a_{lk}^{ji}\in\mathbb{R}$ are the elements of symmetric $nm\times nm$ matrix
\begin{equation}\label{Clmatr}
\left(a_{lk}^{ji}\right), \quad a_{lk}^{ji}=a_{kl}^{ij},\,\,
j,i=\overline{1,n}, \,\, k,l=\overline{0,m-1}.
\end{equation}
This matrix is presented as follows:
$$
\left(a_{lk}^{ji}\right)=
\begin{pmatrix}
A^{11}&A^{12}&\hdotsfor{1}&A^{1n} \\
A^{21}&A^{22}&\hdotsfor{1}&A^{2n} \\
\vdots& \vdots &\ddots & \vdots \\
A^{n1}&A^{n2}&\hdotsfor{1}&A^{nn} \\
\end{pmatrix},$$
where
$$
A^{ji}=\begin{pmatrix}
a_{00}^{ji}& a_{01}^{ji} & \hdotsfor{1} & a_{0(m-1)}^{ji} \\
a_{10}^{ji}& a_{11}^{ji} & \hdotsfor{1} & a_{1(m-1)}^{ji} \\
\vdots& \vdots &\ddots & \vdots \\
a_{(m-1)0}^{ji}& a_{(m-1)1}^{ji} & \hdotsfor{1} & a_{(m-1)(m-1)}^{ji} \\
\end{pmatrix},\,\,\, i,j=\overline{1,n}.
$$
Multiplying $\boldsymbol{a}_j^p$ with $\boldsymbol{a}_i^q$ and replacing products $a_l^ja_k^i$ by the elements $a_{lk}^{ji}$ of matrix (\ref{Clmatr}) we get the following numbers of algebra $\mathcal{A}$:
\begin{equation}\label{Aji}
\boldsymbol{a}_{ji}^{pq}=
\sum\limits_{l,k=0}^{m-1}\eta_{lp}\eta_{kq}a_{lk}^{ji}\boldsymbol{e}^{-1}_l\boldsymbol{e}^{-1}_k,
\,\, j,i=\overline{1,n},\,\,   p,q=\overline{0,m-1}.
\end{equation}
Then the quadratic  form can be expressed  in terms of numbers $\boldsymbol{z}_j^{\boldsymbol{p}}$, $\boldsymbol{z}_i^{\boldsymbol{q}}$,   $\boldsymbol{a}_{ji}^{pq}$ as follows:
\begin{multline*}
\sum\limits_{j,i=1}^n\sum\limits_{l,k=0}^{m-1}
a_{lk}^{ji}x_l^jx_k^i=\sum\limits_{j,i=1}^n\sum\limits_{l,k=0}^{m-1}
a_{lk}^{ji}\left(\boldsymbol{e}^{-1}_l\sum\limits_{p=0}^{m-1}\eta_{lp}\boldsymbol{z}_j^{\boldsymbol{p}}\right)x_k^i=\\
=\sum\limits_{j,i=1}^n\sum\limits_{l,k=0}^{m-1}a_{lk}^{ji}\boldsymbol{e}^{-1}_l\sum\limits_{p=0}^{m-1}\eta_{lp}x_k^i\boldsymbol{z}_j^{\boldsymbol{p}}=
\sum\limits_{j,i=1}^n\sum\limits_{l,k=0}^{m-1}a_{lk}^{ji}\boldsymbol{e}^{-1}_l\sum\limits_{p=0}^{m-1}\eta_{lp}\left(\boldsymbol{e}^{-1}_k
\sum\limits_{q=0}^{m-1}\eta_{kq}\boldsymbol{z}_i^{\boldsymbol{q}}\right)\boldsymbol{z}_j^{\boldsymbol{p}}=\\
=\sum\limits_{j,i=1}^n\sum\limits_{p,q=0}^{m-1}\sum\limits_{l,k=0}^{m-1}\eta_{lp}\eta_{kq}a_{lk}^{ji}\boldsymbol{e}^{-1}_l\boldsymbol{e}^{-1}_k\boldsymbol{z}_i^{\boldsymbol{q}}\boldsymbol{z}_j^{\boldsymbol{p}}=
\sum\limits_{j,i=1}^n\sum\limits_{p,q=0}^{m-1}\boldsymbol{a}_{ji}^{pq}\boldsymbol{z}_i^{\boldsymbol{q}}\boldsymbol{z}_j^{\boldsymbol{p}}.
\end{multline*}

Let $\rho(\boldsymbol{z})=\rho(z): \mathbb{R}^{mn}\rightarrow\mathbb{R}$ have continuous partial derivatives of the first and the second order at a point $w\in\mathbb{R}^{mn}$. Then function $\rho(z)$ is twice continuously differentiable at the point $w$  and its full differentials of the first and the second order are defined as follows:
$$
d\rho(w)=\sum\limits_{j=1}^n
\sum\limits_{l=0}^{m-1}\frac{\partial\rho(w)}{\partial
x^j_l}\,dx^j_l,\qquad
d^2\rho(w)=\sum\limits_{j,i=1}^n\sum\limits_{l,k=0}^{m-1}\frac{\partial^2\rho({w})}{{\partial
x^i_k\partial x^j_l}}\,dx^j_l\,dx^i_k.
$$
We present $d\rho(w)$, $d^2\rho(w)$ in terms of the elements of algebra $\mathcal{A}$. Let
$$
d\boldsymbol{z}_j^{\boldsymbol{p}}:=\sum\limits_{l=0}^{m-1}\gamma_{pl}dx^j_l\boldsymbol{e}_l,\quad \mbox{where} \,\, \gamma_{pl}=1 \,\, \mbox{as}\,\,p=0.
$$
Let $a_l^j=\dfrac{\partial\rho(w)}{\partial x_{l}^j}$, $\boldsymbol{a}_j^p=\dfrac{\partial\rho(\boldsymbol{w})}{\partial
\boldsymbol{z}^{\boldsymbol{p}}_j}$ in (\ref{Ajp}) and $a_{lk}^{ji}=\dfrac{\partial^2\rho(w)}{{\partial x^j_l\partial x^i_k}}$, $\boldsymbol{a}_{ji}^{pq}=\dfrac{\partial^2\rho(\boldsymbol{w})}{{\partial \boldsymbol{z}_j^{\boldsymbol{p}}\partial
\boldsymbol{z}_i^{\boldsymbol{q}}}}$ in (\ref{Aji}), $p,q=\overline{0,m-1}$. Then
\begin{equation}\label{L11}
\dfrac{\partial\rho(\boldsymbol{w})}{\partial
\boldsymbol{z}^{\boldsymbol{p}}_j}:=\sum\limits_{l=0}^{m-1}\eta_{lp}\dfrac{\partial\rho(w)}{\partial x_{l}^j}\boldsymbol{e}^{-1}_l, \,\, j=\overline{1,n},\,\,   p=\overline{0,m-1},
\end{equation}
\begin{equation}\label{L111}
\frac{\partial^2\rho(\boldsymbol{w})}{{\partial \boldsymbol{z}_j^{\boldsymbol{p}}\partial
\boldsymbol{z}_i^{\boldsymbol{q}}}}:=
\sum\limits_{l,k=0}^{m-1}\eta_{lp}\eta_{kq}\frac{\partial^2\rho(w)}{{\partial x^j_l\partial x^i_k}}\boldsymbol{e}^{-1}_l\boldsymbol{e}^{-1}_k,
\,\, j,i=\overline{1,n},\,\,   p,q=\overline{0,m-1}.
\end{equation}
And
\begin{equation}\label{L1}
d\rho(\boldsymbol{w})=\sum\limits_{j=1}^n
\sum\limits_{p=0}^{m-1}\frac{\partial\rho(\boldsymbol{w})}{\partial
\boldsymbol{z}_j^{\boldsymbol{p}}}\,d\boldsymbol{z}_j^{\boldsymbol{p}},\qquad
d^2\rho(\boldsymbol{w})=\sum\limits_{j,i=1}^n\sum\limits_{q,p=0}^{m-1}\frac{\partial^2\rho(\boldsymbol{w})}{{\partial \boldsymbol{z}_j^{\boldsymbol{p}}\partial
\boldsymbol{z}_i^{\boldsymbol{q}}}}\,d\boldsymbol{z}_j^{\boldsymbol{p}}\,d\boldsymbol{z}_i^{\boldsymbol{q}}.
\end{equation}

On the other hand,  substitute the values of  $x^j_l$, $j=\overline{1,n}$, $l=\overline{0,m-1}$, from (\ref{x}) in the expression of function $\rho(z)=\rho(x^1_0,x^1_1,\ldots, x^n_{m-1})$. We obtain $\rho(z)=\rho(x^1_0(\boldsymbol{z}_1^{\boldsymbol{0}},\boldsymbol{z}_1^{\boldsymbol{1}}\ldots,\boldsymbol{z}_1^{\boldsymbol{m-1}}),x^1_1
(\boldsymbol{z}_1^{\boldsymbol{0}},\boldsymbol{z}_1^{\boldsymbol{1}}\ldots,\boldsymbol{z}_1^{\boldsymbol{m-1}}),\ldots, x^n_{m-1}(\boldsymbol{z}_n^{\boldsymbol{0}},\boldsymbol{z}_n^{\boldsymbol{1}}\ldots,\boldsymbol{z}_n^{\boldsymbol{m-1}}))$. Formally differentiating  function $\rho$ as a composite function with respect to variables $\boldsymbol{z}_j^{\boldsymbol{l}}$, $j=\overline{1,n}$, $l=\overline{0,m-1}$, we also obtain formulas (\ref{L11}), (\ref{L111}) for the formal partial derivatives $\dfrac{\partial\rho(\boldsymbol{w})}{\partial
\boldsymbol{z}^{\boldsymbol{p}}_j}$, $\dfrac{\partial^2\rho(\boldsymbol{w})}{{\partial \boldsymbol{z}_j^{\boldsymbol{p}}\partial
\boldsymbol{z}_i^{\boldsymbol{q}}}}$.

In the case when $\mathrm{dim}\,\mathcal{ A}=2^k$, $k\in \mathbb{N}$, it is possible to fit matrix $\Gamma$ such that  $|\gamma_{lq}|=1$ and $\Gamma^{-1}=\dfrac{1}{2^k}\Gamma$:
\begin{equation*}\label{Mx}
\begin{array}{l}
\Gamma_{1}=\left(\begin{array}{rr}
1&1 \\
1&-1
\end{array}\right),\,\,
\Gamma_{2}=\left(\begin{array}{rr}
\Gamma_1&\Gamma_1 \\
\Gamma_1&-\Gamma_1
\end{array}\right),\,\,\ldots \\ \\ \ldots,
\Gamma=\Gamma_{k}=\begin{pmatrix}
\Gamma_{k-1}&\Gamma_{k-1} \\
\Gamma_{k-1}&-\Gamma_{k-1}
\end{pmatrix}=
\begin{pmatrix}
1&1&\hdotsfor{1}&1&1 \\
1&-1&\hdotsfor{1}&1&-1 \\
\vdots& \vdots &\ddots & \vdots & \vdots \\
1&1&\hdotsfor{1}&(-1)^{k-1}&(-1)^{k-1} \\
1&-1&\hdotsfor{1}&(-1)^{k-1}&(-1)^k \\
\end{pmatrix}.
\end{array}
\end{equation*}

Thus, it is not difficult to see that matrix $\Gamma_1$ corresponds to the case of algebra of complex numbers $\mathbb{C}$ \cite{Zin1971} and formula (\ref{L11}) gives a generalization of well known complex formal derivatives $\dfrac{\partial\varphi}{\partial z}$, $\dfrac{\partial\varphi}{\partial \bar{z}}$, $z\in\mathbb{C}$,  to the algebra $\mathcal{A}$.

\section{Generalized linear convexity}
Let $\boldsymbol{s}_j=\sum\limits_{l=0}^{m-1}s_{l}^j\boldsymbol{e}_l$, $j=\overline{1,n}$. We say that a hyperplane
\begin{multline}\label{eque2}
\Pi_{\mathcal{A}}:=\left\{\boldsymbol{s}=(\boldsymbol{s}_1,\boldsymbol{s}_2,\ldots,\boldsymbol{s}_n)\in \mathcal{A}^n:
\sum\limits_{j=1}^n\boldsymbol{c}_j\boldsymbol{s}_j=\boldsymbol{0},\right.\\
\Bigg.(\boldsymbol{c}_1,\boldsymbol{c}_2,\ldots ,\boldsymbol{c}_n)\in\mathcal{A}^n\setminus\{\boldsymbol{0}\}\Bigg\},
\end{multline}
\emph{\textbf{lies in a real hyperplane}}
\begin{multline}\label{eque3}
\Pi_{\mathbb{R}}:=\left\{\left(s^1_{0},s^1_{1},\ldots ,s^n_{(m-1)}\right)\in \mathbb{R}^{mn}:
\sum\limits_{j=1}^n\sum\limits_{l=0}^{m-1}a_{l}^js_{l}^j=0,\right.\\
\Bigg.\left(a^1_{0},a^1_{1},\ldots ,a^n_{(m-1)}\right)\in\mathbb{R}^{mn}\setminus\{0\}\Bigg\},
\end{multline}
if any vector $\boldsymbol{s}$ satisfying the equation of the hyperplane (\ref{eque2}) satisfies the equation of the hyperplane (\ref{eque3}).
\begin{lemma}\label{realhyper}
For any real hyperplane $\Pi_{\mathbb{R}}$ (\ref{eque3}), the hyperplane $\Pi_{\mathcal{A}}$ such that $\boldsymbol{c}_j=\sum\limits_{k,l=0}^{m-1}\eta_{kl}^{\tilde{p}}a^j_{l}\boldsymbol{e}_k$, $j=\overline{1,n}$, where  $\eta_{kl}^{\tilde{p}}$ are the elements of matrix inverse to the matrix $(\gamma_{lk}^{\tilde{p}})$ satisfying condition 2), lies in $\Pi_{\mathbb{R}}$.
\end{lemma}
\begin{proof}[Proof.]
Consider the structure constants $\gamma_{lk}^p$ of $\mathcal{A}$ (\ref{struct}).
Suppose $\mathrm{det}\,(\gamma_{lk}^{\tilde{p}})\ne 0$ for some $p=\tilde{p}$. Let arbitrary real numbers $a^j_{l}$, $l=\overline{0,m-1}$,  $j=\overline{1,n}$, not equal to zero simultaneously and defining $\Pi_{\mathbb{R}}$ be given.
Let
\begin{equation}\label{eque4}
c^j_{k}=\sum\limits_{l=0}^{m-1}\eta_{kl}^{\tilde{p}}a^j_{l},\quad   k=\overline{0,m-1},\,\, j=\overline{1,n},
\end{equation}
where $(\eta_{kl}^{\tilde{p}})=(\gamma_{lk}^{\tilde{p}})^{-1}$.
Hence
$$
a^j_{l}=\sum\limits_{k=0}^{m-1}\gamma_{lk}^{\tilde{p}}c^j_{k},\quad  l=\overline{0,m-1}, j=\overline{1,n}.
$$
Let
\begin{equation}\label{eque1}
\boldsymbol{c}_j=\sum\limits_{k=0}^{m-1}c^j_{k}\boldsymbol{e}_k=\sum\limits_{k,l=0}^{m-1}\eta_{kl}^{\tilde{p}}a^j_{l}\boldsymbol{e}_k, \,\, j=\overline{1,n}.
\end{equation}
Substitute the values of $\boldsymbol{c}_j$ from the first equality of (\ref{eque1}) in the equation of the hyperplane (\ref{eque2}) considering (\ref{struct})
\begin{multline*}
\sum\limits_{j=1}^n\boldsymbol{c}_j\boldsymbol{s}_j=\sum\limits_{j=1}^n\sum\limits_{k=0}^{m-1}c^j_{k}\boldsymbol{e}_k\sum\limits_{l=0}^{m-1}s^j_{l}\boldsymbol{e}_l=\sum\limits_{j=1}^n\sum\limits_{k,l=0}^{m-1}
c^j_{k}s^j_{l}\boldsymbol{e}_k\boldsymbol{e}_l=\\=\sum\limits_{j=1}^n\sum\limits_{k,l=0}^{m-1}c^j_{k}s^j_{l}\sum\limits_{p=0}^{m-1}\gamma_{lk}^p\boldsymbol{e}_p=\sum\limits_{p=0}^{m-1}\sum\limits_{j=1}^n\sum\limits_{k,l=0}^{m-1}\gamma_{lk}^pc^j_{k}s^j_{l}\boldsymbol{e}_p=\boldsymbol{0}.
\end{multline*}
The last equation is equivalent to simultaneous real equations
$$
\sum\limits_{j=1}^n\sum\limits_{k,l=0}^{m-1}\gamma_{lk}^pc^j_{k}s^j_{l}=0,\,\, p=\overline{0,m-1}.
$$
In particular, for $p=\tilde{p}$ we obtain
$$
\sum\limits_{j=1}^n\sum\limits_{k,l=0}^{m-1}\gamma_{lk}^{\tilde{p}}c^j_{k}s^j_{l}=\sum\limits_{j=1}^n\sum\limits_{l=0}^{m-1}a^j_{l}s^j_{l}=0.
$$
Thus,  any vector $\boldsymbol{s}=(\boldsymbol{s}_1,\boldsymbol{s}_2,\ldots,\boldsymbol{s}_n)\in \mathcal{A}^n$ satisfying equation (\ref{eque2}), where constants $\boldsymbol{c}_j$, $j=\overline{1,n}$, are defined by (\ref{eque1}), satisfies equation in (\ref{eque3}).
\end{proof}

Let $\Omega$ be a domain in the space $\mathcal{A}^n$. A domain  is an open connected set in $\mathbb{R}^{mn}$.
\begin{definition}
A domain $\Omega\subset \mathcal{A}^n$ is said to be \emph{\textbf{
locally $\mathcal{A}$-linearly convex}},   if for every boundary point
$\boldsymbol{w}=(\boldsymbol{w}_1,\boldsymbol{w}_2,\ldots, \boldsymbol{w}_n)\in\partial\Omega$ there is a hyperplane $\Pi_{\mathcal{A}}$ (\ref{eque2}), where $\boldsymbol{s}_j=\boldsymbol{z}_j-\boldsymbol{w}_j$,  $\boldsymbol{z}_j\in\mathcal{A}$, $j=\overline{1,n}$, not intersecting $\Omega$ in some neighborhood of the point $\boldsymbol{w}$. The hyperplane $\Pi_{\mathcal{A}}$ is called \emph{\textbf{
locally supporting}} for $\Omega$ at $\boldsymbol{w}$.
\end{definition}

It is obvious that the notion of $\mathbb{C}$-linear convexity is equivalent to the notion of linear convexity.

Now we consider domain
$\Omega=\{\boldsymbol{z}\in \mathcal{A}^n:\rho(\boldsymbol{z})=\rho(\boldsymbol{z},\boldsymbol{z}^{\boldsymbol{1}},\ldots,\boldsymbol{z}^{\boldsymbol{m-1}})<0\}$,
$\boldsymbol{z}^{\boldsymbol{l}}=(\boldsymbol{z}^{\boldsymbol{l}}_1,\boldsymbol{z}^{\boldsymbol{l}}_1,\boldsymbol{z}^{\boldsymbol{l}}_2\ldots \boldsymbol{z}^{\boldsymbol{l}}_n)$, $l=\overline{1,m-1}$, with the boundary
$\partial\Omega =\{\boldsymbol{z}\in \mathcal{A}^n:\rho(\boldsymbol{z})=0\}$, where function $\rho:
\mathcal{A}^n\to\mathbb{R}$ is twice continuously differentiable in a neighborhood of $\partial \Omega$ with respect to its real variables and such that $\mathrm{grad}\rho\ne
0$ everywhere on $\partial\Omega$.

Let $\boldsymbol{w}\in\partial\Omega$, $\boldsymbol{z}_j\in\mathcal{A}$, $\boldsymbol{s}_j=\boldsymbol{z}_j-\boldsymbol{w}_j$,  $j=\overline{1,n}$. We say that vector
$\boldsymbol{s}=(\boldsymbol{s}_1,\boldsymbol{s}_2,\ldots ,\boldsymbol{s}_n)\in \mathcal{A}^n$  \emph{\textbf{ belongs to the tangent hyperplane}}
 $T_{\mathcal{A}}(\boldsymbol{w})$  to $\Omega$ at the point $\boldsymbol{w}$ if
$$
\sum\limits_{j=1}^n\sum\limits_{k,l=0}^{m-1}\eta_{kl}^{\tilde{p}}\frac{\partial\rho(w)}{\partial x^j_l}\boldsymbol{e}_k\,\boldsymbol{s}_j=\boldsymbol{0}.
$$
Thus, if $\boldsymbol{s}\in T_{\mathcal{A}}(\boldsymbol{w})$, then by Lemma \ref{realhyper} and considering (\ref{L1}),
\begin{equation}\label{Re}
\sum\limits_{j=1}^n
\sum\limits_{l=0}^{m-1}\frac{\partial\rho(w)}{\partial
x^j_l}\,s^j_{l}=\sum\limits_{j=1}^n\sum\limits_{l=0}^{m-1}\frac{\partial\rho(\boldsymbol{w})}{\partial
\boldsymbol{z}_j^{\boldsymbol{l}}}\,\boldsymbol{s}_j^{\boldsymbol{l}}=0.
\end{equation}

\begin{theorem}\label{clif} If domain $\Omega$ is locally $\mathcal{A}$-linearly convex and $T_{\mathcal{A}}(\boldsymbol{w})$ is locally supporting for $\Omega$ at any point $\boldsymbol{w}\in\partial\Omega$,  then for any point $\boldsymbol{w}$ and any vector $\boldsymbol{s}\in T_{\mathcal{A}}(\boldsymbol{w})$, $\|\boldsymbol{s}\|=1$, the following inequality is true
\begin{equation}\label{avtoref6}
\sum\limits_{i,j=1}^n\sum\limits_{k,l=0}^{m-1}
\frac{\partial^2\rho(\boldsymbol{w})}{{\partial \boldsymbol{z}^{\boldsymbol{k}}_i\partial
\boldsymbol{z}^{\boldsymbol{l}}_j}}\boldsymbol{s}^{\boldsymbol{l}}_j\boldsymbol{s}^{\boldsymbol{k}}_i\ge 0.
\end{equation}
If for any point  $\boldsymbol{w}\in
\partial\Omega$ and any vector $\boldsymbol{s}\in T_{\mathcal{A}}(\boldsymbol{w})$, $\|\boldsymbol{s}\|=1$,
\begin{equation}\label{4}
\sum\limits_{i,j=1}^n\sum\limits_{k,l=0}^{m-1}
\frac{\partial^2\rho(\boldsymbol{w})}{{\partial \boldsymbol{z}^{\boldsymbol{k}}_i\partial
\boldsymbol{z}^{\boldsymbol{l}}_j}}\boldsymbol{s}^{\boldsymbol{l}}_j\boldsymbol{s}^{\boldsymbol{k}}_i> 0,
\end{equation}
then domain $\Omega$ is locally $\mathcal{A}$-linearly convex.
\end{theorem}

\begin{proof}[Proof. Sufficiency]  Formally write the Taylor series for the function $\rho(\boldsymbol{z})=\rho(\boldsymbol{z}^{\boldsymbol{0}},\boldsymbol{z}^{\boldsymbol{1}},\ldots,\boldsymbol{z}^{\boldsymbol{m-1}})$, $\boldsymbol{z}^{\boldsymbol{l}}=(\boldsymbol{z}^{\boldsymbol{l}}_1,\boldsymbol{z}^{\boldsymbol{l}}_1,\boldsymbol{z}^{\boldsymbol{l}}_2\ldots \boldsymbol{z}^{\boldsymbol{l}}_n)$, $l=\overline{0,m-1}$, with respect to variables $\boldsymbol{z}^{\boldsymbol{l}}_j$ in the neighborhood $U(\boldsymbol{w})$ of any point $\boldsymbol{w}\in\partial\Omega$:
\begin{multline*}
\rho(\boldsymbol{z})=\rho(\boldsymbol{{w}})+\sum\limits_{j=1}^n\sum\limits_{k=0}^{m-1}\frac{\partial\rho(\boldsymbol{w})}{\partial
\boldsymbol{z}_j^{\boldsymbol{k}}}\,(\boldsymbol{z}_j^{\boldsymbol{k}}-\boldsymbol{w}_j^{\boldsymbol{k}}) +\\
+\frac{1}{2}\sum\limits_{i,j=1}^n\sum\limits_{k,l=0}^{m-1}\frac{\partial^2\rho(\boldsymbol{w})}{{\partial
\boldsymbol{z}_i^{\boldsymbol{k}}\partial
\boldsymbol{z}_j^{\boldsymbol{l}}}}\,(\boldsymbol{z}_j^{\boldsymbol{l}}-\boldsymbol{w}_j^{\boldsymbol{l}})
(\boldsymbol{z}_i^{\boldsymbol{k}}-\boldsymbol{w}_i^{\boldsymbol{k}})+o(\|\boldsymbol{z}-\boldsymbol{w}\|^2),\qquad \boldsymbol{z}\to \boldsymbol{w}.
\end{multline*}

Since $\rho(\boldsymbol{w})=0$ at any boundary point $\boldsymbol{w}$ and considering condition (\ref{Re}),
we get
\begin{multline}\label{6}
\rho(\boldsymbol{z})=\frac{1}{2}\left(\sum\limits_{i,j=1}^n\sum\limits_{k,l=0}^{m-1}
\frac{\partial^2\rho({\boldsymbol{w}})}{{\partial \boldsymbol{z}_i^{\boldsymbol{k}}\partial
\boldsymbol{z}_j^{\boldsymbol{l}}}}\,\frac{(\boldsymbol{z}_j^{\boldsymbol{l}}-\boldsymbol{w}_j^{\boldsymbol{l}})
(\boldsymbol{z}_i^{\boldsymbol{k}}-\boldsymbol{w}_i^{\boldsymbol{k}})}{\|\boldsymbol{z}-\boldsymbol{w}\|^2}\right)\|\boldsymbol{z}-\boldsymbol{w}\|^2+\\
+o(\|\boldsymbol{z}-\boldsymbol{w}\|^2), \,\, \boldsymbol{z}\to \boldsymbol{w},
\end{multline}
for any point  $\boldsymbol{z}\in U(\boldsymbol{w})\cap T_{\mathcal{A}}(\boldsymbol{w})$.

Thus,  $\rho(\boldsymbol{z})\ge 0$ for any point $\boldsymbol{z}\in U(\boldsymbol{w})\cap
T_{\mathcal{A}}(\boldsymbol{w})$ and any point $\boldsymbol{w}\in\partial\Omega$  by  (\ref{4}) and (\ref{6}), which means local $\mathcal{A}$-linear convexity of domain  $\Omega$.

\vskip 1mm {\it Necessity.} Let domain $\Omega$ be locally $\mathcal{A}$-linearly convex and for a point
$\widetilde{\boldsymbol{w}}=(\widetilde{\boldsymbol{w}}_1,\widetilde{\boldsymbol{w}}_2,\ldots ,\widetilde{\boldsymbol{w}}_n)\in\partial\Omega$ and for a vector  $\boldsymbol{t}=(\boldsymbol{t}_1,\boldsymbol{t}_2,\ldots ,\boldsymbol{t}_n)\in T_{\mathcal{A}}(\widetilde{\boldsymbol{w}})$ the following inequality is true
\begin{equation} \label{5}
\sum\limits_{i,j=1}^n\sum\limits_{k,l=0}^{m-1}
\frac{\partial^2\rho(\widetilde{\boldsymbol{w}})}{{\partial \boldsymbol{z}^{\boldsymbol{k}}_i\partial
\boldsymbol{z}^{\boldsymbol{l}}_j}}\boldsymbol{t}^{\boldsymbol{l}}_j\boldsymbol{t}^{\boldsymbol{k}}_i< 0.
\end{equation}
On the other hand, for points $\boldsymbol{z}\in U(\widetilde{\boldsymbol{w}})\cap T_{\mathcal{A}}(\widetilde{\boldsymbol{w}})$
the expansion (\ref{6}) is valid. Thus, for the point
$\widetilde{\boldsymbol{z}}=(\widetilde{\boldsymbol{z}}_1,\widetilde{\boldsymbol{z}}_2,\ldots ,\widetilde{\boldsymbol{z}}_n)\in U(\widetilde{\boldsymbol{w}})\cap
T_{\mathcal{A}}(\widetilde{\boldsymbol{w}})$, which corresponds to the tangent vector $\boldsymbol{t}$, where
correspondence  is defined by the relation
$\boldsymbol{t}_i=(\widetilde{\boldsymbol{z}}_i-\widetilde{\boldsymbol{w}}_i)/\|\widetilde{\boldsymbol{z}}-\widetilde{\boldsymbol{w}}\|$, $i=\overline{1,n}$, the inequality $\rho(\widetilde{\boldsymbol{z}})<0$ is true by (\ref{5}), which contradicts the fact that hyperplane $T_{\mathcal{A}}(\widetilde{\boldsymbol{w}})$ is locally supporting for $\Omega$ at $\widetilde{\boldsymbol{w}}$.
\end{proof}

\end{document}